\documentclass[11pt,letterpaper]{article}
\usepackage{amsfonts, amsmath, amssymb, amscd, amsthm,
tikz-cd}

\numberwithin{equation}{section}

\let\OLDthebibliography\thebibliography
\renewcommand\thebibliography[1]{
  \OLDthebibliography{#1}
  \setlength{\parskip}{1.5pt}
  \setlength{\itemsep}{1.5pt plus 0.3ex}
}

\usepackage{hyperref}
\usepackage{tocloft}
\hypersetup{linktocpage}

\hypersetup{colorlinks,
    linkcolor={red!50!black},
    citecolor={blue!80!black},
    urlcolor={blue!80!black}}

\renewcommand{\phi}{\varphi}

\newcommand{\K}{\mathcal K}
\renewcommand{\L}{\mathcal L}

\newcommand{\NN}{\mathbb{N}}
\newcommand{\ZZ}{\mathbb{Z}}

\newcommand{\e}{\varepsilon}

\renewcommand{\ll }{\left\langle\hspace{-.7mm}\left\langle }
\newcommand{\la}{\langle}
\newcommand{\ra}{\rangle}
\newcommand{\rr }{\right\rangle\hspace{-.7mm}\right\rangle }

\newcommand{\Hl}{\{ H_\lambda\}_{\lambda\in \Lambda}}

\newtheorem{thm}{Theorem}[section]
\newtheorem*{thm*}{Theorem}
\newtheorem{cor}[thm]{Corollary}
\newtheorem{lem}[thm]{Lemma}
\newtheorem{prop}[thm]{Proposition}

\theoremstyle{definition}
\newtheorem{defn}[thm]{Definition}

\theoremstyle{remark}

\def\pres<#1|#2>{\gp{#1\ \vrule\ #2}}
\let\le\leqslant

\let\texthat\^%
\def\^#1{\ifmmode\widehat{#1}\else\texthat#1\fi}
\def\gp#1{\left\langle#1\right\rangle}
\def\nc#1{\gp{\!\gp{#1}\!}}
\let\iso\simeq
\let\zvezda*
\newbox\tmpbox
\newdimen\tmpdim
\def\narrow[#1]#2\par 
{%
\setbox\tmpbox\hbox{#2}%
\ifdim\wd\tmpbox>#1
\setbox\tmpbox\vbox{\hsize=#1 #2}%
\tmpdim=\ht\tmpbox%
\loop%
\setbox\tmpbox\vbox{\hsize=\wd\tmpbox \advance\hsize by -1pt
#2}%
\ifdim\ht\tmpbox=\tmpdim%
\relax%
\repeat%
\setbox\tmpbox\vbox{\hsize=\wd\tmpbox \advance\hsize by 1pt #2}%
\fi%
\box\tmpbox
}
\def\disp#1{
                $$
                \setbox\tmpbox\vbox{\narrow[\hsize]\noindent#1\par}
                \box\tmpbox
                $$
          }

\def\dispno#1(#2){    
                $$
                \setbox\tmpbox\vbox{\narrow[\hsize]\noindent#1\par}
                \vcenter{\box\tmpbox}
                                \eqno{(#2)}
                $$
          }

\begin{document}

\title{Equationally separable classes of groups}

\bigskip

\author{%
Alexander A. Buturlakin%
\thanks{The first author was supported by the Program of Fundamental
Research of the Russian Academy of Sciences, project~FWNF-2026-0017.},\;
Anton A. Klyachko%
\thanks{The second author was supported by the Russian Science Foundation,
project 22-11-00075.
},
\;
Denis V. Osin%
\thanks{The third author was supported by the NSF grant DMS-2405032 and
the Simons Fellowship in Mathematics MP-SFM-00005907.}}

\date{}

\maketitle

\begin{abstract}
For every nontrivial finite group $G$, there exists a
finite
system of equations with coefficients in $G$ that has no solution in any
finite overgroup of $G$ but has a solution in a periodic overgroup of
$G$. We prove several similar facts about
amenable, orderable, locally indicable, solvable, nilpotent, and other
classes of groups. As a byproduct, we also show that any amalgam of two
countable periodic groups with finite intersection embeds into a periodic
group, thereby answering a 1960 question of B.~Neumann in the countable
case.
\end{abstract}

\section{Introduction}

Let $F_n$ denote the free group of rank $n$ with a basis
$\{x_1, \ldots,x_n\}$.
Recall that a \emph{system of equations} with coefficients in a
group $G$ and unknowns $x_1, \ldots, x_n$ is a collection of equalities of
the form
\begin{equation}\label{Eq:sys}
w_i=1, \;\;\; i\in I,
\end{equation}
where each $w_i$ is an element of $G\ast F_n$. If $n=1$, we call the
system (\ref{Eq:sys}) \emph{univariate}.

A tuple $(g_1, \ldots, g_n)\in G^n$ is a \emph{solution} to the system
(\ref{Eq:sys}) if every $w_i$ lies in the kernel of the
homomorphism $G\ast F_n\to G$ whose restriction to $G$ is the identity
and which sends $x_j$ to $g_j$ for all $j=1, \ldots , n$. Further, the system (\ref{Eq:sys}) is
\emph{solvable over} $G$ if it is solvable in
an overgroup
$\widetilde G \supseteq G$; if such a group $\widetilde G$ can be chosen
in a class of groups $\K$, we say that the system~(\ref{Eq:sys})
is \emph{solvable in~$\K$}.

Research on the solvability of equations over groups can be said to have
begun with Magnus's Freiheitssatz \cite{Mag30} and has remained an active
area of study ever since. The central problem in this context is
determining whether a given equation or system of equations with
coefficients in a group $G$ admits a solution over $G$. Although
this question remains largely open in general, the answer has been shown to be
affirmative in many particular cases. Moreover, if $G$ belongs to a
suitably ``well-behaved" class $\K$ (e.g., the class of finite groups or
locally indicable groups) and the system is ``non-degenerate" in a certain
precise sense, the solution can often be found in an overgroup of $G$ that
also belongs to $\K$.

Classical results in this direction include the theorem of Gerstenhaber
and Rothaus \cite{GR62} on solvability of systems over finite groups and
the Brodskii--Howie--Short theorem on equations over locally indicable
groups \cite{B84,How81,Sh81}. For more recent developments, we refer to
the papers \cite{ABA21,EH21,NT22,KM23,Ch23,KMR24,Mi24,Mi24a,Mi26b,Mi26c,
KMO26}
and
references therein; see also the survey \cite{Ro12}.

In this paper, we propose a general method for constructing systems of
equations that separate natural -- and often closely related -- classes of
groups in the following sense.

\begin{defn}
A subclass $\K$ of a class of groups $\L$ is said to be \emph{equationally
separable} in $\L$ if, for every nontrivial group $G\in \K$, there exists
a finite system of equations with coefficients in $G$, called a
\emph{separating system}, that is solvable in $\L$ but not in $\K$.
\end{defn}

We mention three sample results here and refer the reader to
Theorem~\ref{Thm:main} for more examples.

\begin{thm}\label{Thm:Ex}
\begin{enumerate}
\item[(a)]
The class of all finite groups is equationally separable
in the class of all periodic groups.
\item[(b)]
The class of all groups with decidable conjugacy problem is
equationally separable in the class of
all groups with decidable word problem.
\item[(c)]
The class of all locally indicable groups is equationally separable
in the class of all
left-orderable groups.
\end{enumerate}
\noindent
Moreover, in cases (a) and (b), the separating systems can
be made univariate.
\end{thm}

Our approach relies on an elementary yet useful result (Theorem
\ref{Thm:Simple}), which serves as a meta-theorem for deriving claims as
in the theorem above. The reduction of the number of unknowns to one in
parts (a) and (b) is achieved by utilizing recent advances in the theory
of relatively hyperbolic groups. For details, we refer the reader to
Section \ref{Sec:Red}.

In the course of proving our main results, we also obtain a partial answer
to an old question of B.~Neumann. Recall that an \emph{amalgam $(A,B;C)$
of groups $A$ and $B$ with intersection $C$} is a pair of groups $A$ and
$B$ such that $C=A\cap B$ is a subgroup of both $A$ and $B$ and the
multiplications on $C$ (inherited from $A$ and $B$) coincide.
Of course, any amalgam $(A,B;C)$ embeds into a group (e.g., into
the amalgamated product $A*_CB$).
In
\cite{Neu60}, B. Neumann noted that not every amalgam of periodic groups
embeds into a periodic group and asked whether every amalgam of periodic
groups with finite intersection embeds into a periodic group. In the last
section \ref{Sec:Period}, we obtain an affirmative answer to this question
for countable groups.

\begin{thm}\label{Thm:AmEmb}
For any amalgam $(A,B;C)$ of countable periodic groups with finite
intersection $C$, there exists a periodic group containing $A$ and $B$ as
subgroups such that $A\cap B=C$.
\end{thm}

It is worth noting that modern techniques make the proof of Theorem
\ref{Thm:AmEmb} easy. However, the general case of B. Neumann's problem
(without cardinality restrictions) appears to be significantly harder.

Throughout this paper, we use the following standard notation. For any
$k\in \ZZ$ and elements $x$ and $y$ of a group $G$, we denote by
$x^{ky}$ the element $y^{-1}x^ky$; in particular, we write
$x^y$
and $x^{-y}$
for
$y^{-1}xy$
and $y^{-1}x^{-1}y$, respectively.
The commutator $x^{-1}y^{-1}xy$ is denoted
by $[x,y]$.  By $\la S\ra$ and $\ll S\rr^G$ we denote the subgroup
generated by a set $S\subseteq G$ and the normal closure of~$S$ in $G$,
respectively. A cyclic group
order $n\in \NN\cup \{\infty \}$
generated by an element $x$ is denoted by $\la x\ra _n$.
The symbol $\ZZ_n$ denotes the group $\ZZ/n\ZZ$.

\paragraph{Acknowledgments.}
The authors thank A.Yu. Olshanskii and F. Fournier-Facio for their
valuable remarks. We are grateful to an anonymous referee for pointing out
several imprecisions in an earlier version of this paper.
The second author thanks the Theoretical Physics and Mathematics
Advancement Foundation ``BASIS".

\section{Simple classes of groups}\label{Sec:Simple}

We begin by introducing the notion of a simple class, which plays a
central role in our paper.

\begin{defn}\label{Def:Simple}
Let $\L$ be a class of groups. We say that $\L$ is \emph{simple} if for
any groups $G,H\in \L$, any $g\in G$, and any $h\in H\setminus \{ 1\}$,
there exists a group $K\in\L$ containing $G$ and $H$ as subgroups such
that $g\in\ll h\rr^K$.
\end{defn}

Below we provide some examples of simple classes relevant to our work. In
this paper, we always assume groups with solvable word or conjugacy
problems to be finitely generated. Recall also that a periodic group $G$
is a \emph{$\pi$-group}, where $\pi$ is a set of primes, if the order of
any element of $G$ decomposes as a product of elements of $\pi$. Finally,
we say that a group $G$ is \emph{free-group-free} if $G$ contains no
non-cyclic free subgroups.

\begin{prop}\label{Prop:Simple}
The following classes of groups are simple:
\begin{enumerate}
\item[(a)] all groups;
\item[(b)] groups with decidable word problem;
\item[(c)] torsion-free groups;
\item[(d)] periodic groups and, more generally, $\pi$-groups,
where $\pi$ is an arbitrary set of primes;
\item[(e)] left-orderable groups;
\item[(f)] amenable groups;
\item[(g)] free-group-free groups;
\item[(h)] solvable groups.
\end{enumerate}
\end{prop}

\begin{proof}
In what follows, we use the notations from Definition \ref{Def:Simple}.

\emph{(a)\;}
This follows immediately from the well-known fact that any group (e.g.,
$G\times H$) embeds into a simple group (see, for example, \cite{KaM79},
Section 13).

\emph{(b)\;}
Clearly, the word problem is decidable in $G\times H$ whenever it is
decidable in $G$ and $H$. By the Boone--Higman theorem, $G\times H$ (as
well as any other group with decidable word problem) embeds into a simple
subgroup $K$ of a finitely presented group (see \cite{BH74} or
\cite[Chapter IV, Theorem~7.4]{LS80}). It remains to note that a simple
subgroup of a finitely presented group has decidable word problem by the
same Boone--Higman theorem (in fact, this direction of the theorem is due
to Kuznetsov \cite{Ku58}, see also \cite[Chapter IV, Theorem 3.6]{LS80}).

\emph{(c)\;}
If $g=1$, then we have nothing to prove; otherwise, we can take $K$ to be
the free product with cyclic amalgamation
$K=G\ast_{\la g\ra =\la h\ra }H$.

\emph{(d)\;}
The class of all periodic groups can be treated similarly to (a) by
applying Olshanskii's theorem: \emph{any periodic group (in particular,
$G\times H$) embeds into a simple periodic group} \cite{Ols92}.

It is unknown whether
an analogue of
this theorem holds for $\pi$-groups, so a trickier
argument is needed.  First, note that the diagonal subgroup of the base of
the wreath product $\ZZ_n\wr\ZZ_n$ is contained in the commutator subgroup
of this wreath product. Therefore, we can assume that the element $g\in G$
lies in the commutator subgroup of $G$; indeed, it suffices to replace $G$
with a larger $\pi$-group~$G\wr\ZZ_{n}$, where
$n=|\gp g|$,
containing $G$ as
the diagonal subgroup of the base. Now, we can take $K$ to be the
restricted wreath product $G\wr H$ (containing $G$ as a direct factor of
the base), because the normal closure of any nontrivial subgroup of $H$ in
$G\wr H$ contains the commutator subgroup of each direct factor of the
base (see \cite[Proposition 26.21]{Neu69}).

\emph{(e)\;}
This can be explained similarly to (c) since the class of left-orderable
groups is closed under free products with cyclic amalgamation
\cite[Theorem B]{BG08}.

\emph{(f), (g), (h)\;}
For a given group $G$, consider the group of eventually constant functions
$$
\widehat G= \{f\colon \ZZ\to G\mid
\exists\, n\in\ZZ\; \exists\,
b\in G\;\forall\, m>n\; f(m)=b {\;\rm and\;}  f(-m)=1\}.
$$
We represent elements of $\widehat G$ as bi-infinite sequences
$(\ldots, f(-2), f(-1), f(0); f(1), f(2), \ldots)$, where the semicolon
separates $f(0)$ from $f(1)$. The map $a\mapsto (\ldots,1,1,a;1,1,
\ldots)$ defines an embedding of $G$ into $\widehat G$. Note that there is a
homomorphism ${\widehat G}\to G$ sending $f\in \widehat G$ to its
eventually constant value, whose kernel is the direct sum of copies of
$G$. Therefore, $\widehat G$ is amenable (respectively, free-group-free or
solvable) whenever $G$ is.

Further, consider the semidirect product $\widehat G\rtimes \ZZ$, where
$\ZZ\iso\la z\ra$ acts by right shifts.
Again, it is easy to see that
$\widehat G\rtimes \ZZ$ is amenable (respectively, free-group-free or
solvable) whenever $G$ is. The commutator $[z,(\ldots,1,1,g;g,\ldots)]$
equals $(\ldots,1,1,g; 1,1,\ldots)$. Thus, replacing $G$ with
$\widehat G\rtimes \ZZ$ if necessary, we can assume that the element $g$
from Definition \ref{Def:Simple} lies in $[G,G]$. Now, we can take $K$ to
be the restricted wreath product $G\wr H$ and apply the (easy-to-prove)
fact from \cite{Neu69} mentioned in the proof of (d).
\end{proof}

Recall that a \emph{quasi-identity} is a first-order sentence of the form
\begin{equation}\label{Eq:q}
\forall x,y,\ldots\quad
\Big(w_1(x,y,\ldots)=1\;\&\;\ldots\;\&\;w_n(x,y,\ldots)=1\Big)\;
\Longrightarrow\; v(x,y,\ldots)=1,
\end{equation}
where $w_i$ and $v$ are words in the alphabet $\{x,y,\dots\}^{\pm 1}$. We
say that a quasi-identity holds in a class of groups if it is
satisfied in every group from this class.

Quasi-identities arise naturally in the context of equations over groups;
see \cite{K99, I05, KMR24}. Standard facts about quasi-identities and
quasivarieties can be found in the books \cite{Bu02} and \cite{Go99}.

\begin{thm}\label{Thm:Simple}
Suppose that a quasi-identity (\ref{Eq:q}) does not hold in a simple class
of groups $\L$.
Then,
for every nontrivial group $G\in\L$, there exists a
finite system of equations over $G$ that is solvable in a group
$K\supseteq G$ from $\L$ but not solvable in any overgroup of $G$
satisfying (\ref{Eq:q}).
\end{thm}

\begin{proof}
Suppose that the quasi-identity (\ref{Eq:q}) is violated by elements
$\tilde x,\tilde y,\ldots\in H\in\L$. Consider an element $g\in
G\setminus\{ 1\}$. By the definition of simplicity, applied to
$h=v(\tilde x,\tilde y,\ldots)\ne 1$, there exists a group $K\in\L$
containing $G$ and $H$ as subgroups such that
$
g=\prod_{i=1}^r v(\tilde x,\tilde y,\ldots)^{\pm\tilde z_i}
$
for some $r\in\NN$ and some elements $\tilde z_1,\ldots,\tilde z_r\in K$.
Thus, the system of equations
$$
\left\{w_1(x,y,\ldots)=1,\;\dots,\;w_n(x,y,\dots)=1,\;
\prod\limits_{i=1}^r v(x,y,\dots)^{\pm z_i}=g
\right\}
$$
over $G$ (with unknowns $x,y,\ldots,z_1,\ldots,z_r$) has an obvious solution
in $K\supseteq G$, but cannot have a solution in any group
satisfying the quasi-identity (\ref{Eq:q}).
\end{proof}

The following examples show that neither the assumption of simplicity of
$\L$ nor the condition on the quasi-identity can be dropped.

\begin{itemize}
\item
Let $\L$ be the (non-simple) class $\{\ZZ_2,\ZZ_3\}$. The quasi-identity
$x^3=1\; \Rightarrow\; x=1$ is satisfied in $G=\ZZ_2$ but not in $\ZZ_3$.
However, any system of equations with coefficients in $G$ solvable in $\L$
is solvable already in $G$.

\item
Let $G$ be a simple group. Clearly, the class $\L=\{G\}$ is simple, but
any system of equations with coefficients in $G$ solvable in $\L$ is
solvable in $G$. A less trivial example illustrating the same point can be
obtained by taking $\L=\{ G\}$, where $G$ is a nontrivial
\emph{algebraically closed} group (i.e., a group $G$ containing a solution
to any finite system of equations solvable over $G$, see \cite{Sc51,
LS80}). In particular, Theorem \ref{Thm:Simple} implies that \emph{a
nontrivial algebraically closed group satisfies no nontrivial
quasi-identity.}
\end{itemize}

The following simple observation will be instrumental in applying Theorem
\ref{Thm:Simple}.

\begin{lem}\label{Lem:qi}
Let $\K$ be a class of groups. Suppose that there exists a finitely
presented group $Q$ and an element $q\in Q\setminus \{1\}$ such that $q$
is contained in the kernel of
every homomorphism
from $Q$ to any group
in
$\K$. Then there exists a quasi-identity that holds in every group from
$\K$ but fails in $Q$.
\end{lem}

\begin{proof}
Let $Q=\la x_1,\dots,x_n\mid w_1=1,\dots,\;w_m=1\ra $ and let
$v$ be any word in the alphabet $\{ x_1, \ldots, x_n\}$
representing $q$. The quasi-identity
$
(w_1=1\;\&\dots\;\&\;w_m=1) \,\Rightarrow \, v=1
$
has the required property.
\end{proof}

\section{Reducing the number of unknowns}\label{Sec:Red}

In this section, we assume that the reader is familiar with the notion of
relative hyperbolicity. We accept the definition of a relatively
hyperbolic group given in \cite{Osi06a}, which does not impose any
cardinality restrictions. For a comprehensive exposition of various
(equivalent) definitions in the case of countable groups, we refer to
\cite{Hru}. Most basic facts about relatively hyperbolic groups used below
can be found in \cite{DGO17,Osi06a,Osi07}. The main goal of this section
is to prove the following.

\begin{prop}\label{Prop:reduction}
Let $G$ be a group, $a\in G\setminus\{ 1\}$, and let $\mathcal A=\{ w_i=1
\mid i=1, \ldots, m\}$ be a system of equations with coefficients in $G$
and unknowns $x_1, \ldots, x_n$.  Further, let $\mathcal B$ be the system
with one unknown $t$ obtained from $\mathcal A$ by making the
substitutions
$$
x_i=at^{20i+1}at^{20i+2}\cdots at^{20i+20}, \;\;\; i=1, \ldots, n,
$$
in all the equations of $\mathcal A$. If $\mathcal A$ has a solution in a
group $K$ containing $G$, then $\mathcal B$ has a solution in a group
hyperbolic relative to $K$.
\end{prop}

To prove Proposition \ref{Prop:reduction}, we will need a few auxiliary
results.

The proof of the following lemma makes use of small-cancellation theory
over free products of groups. For unexplained definitions and details, we
refer the reader to  Chapter 5 of \cite{LS80}.

\begin{lem}\label{Lem:GFn}
Let $G$ be a group, let $a\in G\setminus\{1\}$, and let $F_n$ be
the free group with basis
$x_1,\ldots,x_n$. Define a homomorphism
$\iota\colon G\ast F_n\to G\ast\la t\ra$ by $\iota(g)=g$ for every
$g\in G$ and
$$
\iota(x_i)=at^{20i+1}at^{20i+2}\cdots at^{20i+20},
\qquad i=1,\ldots,n.
$$
For every normal subgroup $N\lhd G\ast F_n$ such that $a\notin N$, the
homomorphism
$$
\bar\iota\colon (G\ast F_n)/N\longrightarrow
(G\ast\la t\ra)/\ll\iota(N)\rr^{G\ast\la t\ra}
$$
induced by $\iota$ is injective,
and the latter group is hyperbolic relative to
$\bar\iota ((G\ast F_n)/N)$. In particular, $\iota$ is injective.
\end{lem}

\begin{proof}
For brevity, we denote $G\ast F_n$ and $G\ast\la t\ra$ by $K$ and $H$,
respectively. Let $N\lhd K$ and suppose that $a\notin N$.
The quotient group $H/\ll\iota(N)\rr^H$ is naturally isomorphic to
the quotient of the free product $(K/N)\ast\la t\ra$ by
the normal closure of
the elements
$$
R_i=\bar x_i^{-1}\bar at^{20i+1}\bar at^{20i+2}\cdots
\bar at^{20i+20}, \qquad i=1,\ldots,n,
$$
where $\bar a=aN$ and $\bar x_i=x_iN$. Since $a\notin N$, we have
$\bar a\ne1$. It is easy to see that the symmetrized set of relations
obtained
from $\{ R_1, \ldots, R_n\}$ by cyclically reducing each $R_i$
(if necessary)
and then adding all cyclic shifts and their inverses
satisfies the $C^\prime(1/6)$
small cancellation condition over the free product $(K/N)\ast\la t\ra$.
By \cite[Chapter V, Corollary 9.4]{LS80}, $K/N$ embeds naturally in
$H/\ll\iota(N)\rr^H$. Furthermore, the Greendlinger lemma for small
cancellation quotients of free products
\cite[Chapter V, Theorem 9.3]{LS80}, together with the definition in
\cite{Osi06a}, gives a linear relative isoperimetric inequality for the
group $H/\ll\iota(N)\rr^H$ with respect to  $\bar\iota (K/N)$. Hence,
$H/\ll\iota(N)\rr^H$ is hyperbolic relative to $\bar\iota (K/N)$.
Taking $N=\{1\}$, we obtain that that $\iota$ is injective.
\end{proof}

In the proof of Proposition \ref{Prop:reduction} given below, we use the
following combination theorem for relatively hyperbolic groups. This
theorem was first proved in \cite[Theorem 0.1 (2)]{D03} in the particular
case of finitely generated groups. For the general case, we refer to
\cite[Corollary~1.5]{Osi06c}. (It is worth noting that the definition of
relative hyperbolicity via relative isoperimetric functions suggested in
\cite{Osi06a} makes this result essentially trivial.)

\begin{lem}\label{Lem:CT}
If a group $A$ is hyperbolic relative to a subgroup $B$ and $B\le C$, then
$A\ast_B C$ is hyperbolic relative to $C$.
\end{lem}

We are now ready to prove the main result of this section.

\begin{proof}[Proof of Proposition \ref{Prop:reduction}]
For every $k=1, \ldots, m$,
let $u_k\in G\ast \ZZ$ be the image of $w_k$ under the map
$\iota\colon G\ast F_n\to G\ast \ZZ$ provided by Lemma \ref{Lem:GFn}.
Recall that $u_k$ is obtained from $w_k$ by substituting
$at^{20i+1}at^{20i+2}\cdots at^{20i+20}$ for $x_i$ in $w_k$ for all
$i=1, \ldots, n$. Thus, $\mathcal B$ consists of the equations $u_k=1$,
$k=1, \ldots, m$, with unknown $t$.

Suppose that $\mathcal A$ has a solution in a group $K$ containing $G$.
Let $\varphi\colon G\ast F_n\to K$ be the homomorphism associated with
this solution; that is, $\varphi$ restricts to the identity on $G$ and
$\varphi(w_k)=1$ for every $k$. Set
$$
N=\ker\varphi \qquad\text{and}\qquad Q=(G\ast F_n)/N.
$$
Then $Q\cong\varphi(G\ast F_n)\le K$, and we identify $Q$ with this
subgroup of $K$. Since $\varphi(a)=a\ne1$, we have $a\notin N$. By Lemma
\ref{Lem:GFn}, the group $Q$ embeds naturally in
$$
P=(G\ast\ZZ)/\ll\iota(N)\rr^{G\ast\ZZ},
$$
and $P$ is hyperbolic relative to $Q$. Moreover, $w_k\in N$ and hence
$u_k=\iota(w_k)$ is trivial in $P$ for every $k$. By Lemma
\ref{Lem:CT}, the amalgamated free product $P\ast_Q K$ is hyperbolic
relative to $K$. Clearly, the image of $t$ in $P\ast_Q K$ provides a
solution to $\mathcal B$.
\end{proof}

\section{Embedding amalgams in periodic groups}\label{Sec:Period}

Our next goal is to prove a general result -- Corollary \ref{Cor:T} below
-- which
is used in the proof of the existence of univariate separating systems
in part (a) of Theorem~\ref{Thm:Ex} and also
implies the solution to B. Neumann's problem mentioned in the
introduction. Although several similar results are known (see, for
example, \cite{Cou,P01}), Corollary \ref{Cor:T} does not seem to have been
recorded in the literature.

\begin{lem}\label{Lem:Q}
Let $G$ be a countable group hyperbolic relative to a collection of
subgroups $\Hl$. For any $g\in G$, there exists an epimorphism $\e\colon
G\to Q$ such that the restriction of $\e$ to
$\bigcup_{\lambda\in \Lambda} H_\lambda$ is injective, $Q$ is hyperbolic
relative to $\{ \e(H_\lambda)\}_{\lambda\in \Lambda}$, and $\e(g)$ either
has finite order or is conjugate to an element of $\e(H_\lambda)$ for some
$\lambda\in \Lambda $ in $Q$.
\end{lem}

\begin{proof}
If $g$ has finite order in $G$ or is conjugate to an element of
$H_\lambda$ for some $\lambda\in \Lambda $, we let $Q=G$. Henceforth, we
can assume that the order of $g$ is infinite. By
\cite[Corollary~1.7]{Osi06a}, $g$ is contained in a virtually cyclic
subgroup $E(g)\le G$ such that $G$ is hyperbolic relative to the
collection of subgroups $\Hl\cup\{E(g)\}$. In particular, in the
terminology of \cite{DGO17}, the subgroup $E(g)$ is hyperbolically
embedded in $G$ with respect to a generating set $X$ that contains the
union $\mathcal H=\bigcup_{\lambda\in \Lambda} H_\lambda$ (see
\cite[Remark 4.26]{DGO17}). Combining the main result of \cite{Osi07} and
part (c) of the more general \cite[Theorem 7.15]{DGO17} yields a finite
subset $\mathcal F\subseteq E(g)\setminus \{1\}$ such that, for any
$N\lhd E(g)$ satisfying $N\cap \mathcal F=\emptyset$, the following hold.
\begin{enumerate}
\item[(a)]
The quotient group $E(g)/N$ and all subgroups $H_\lambda$ naturally embed
in $G/\ll N\rr^G$, and $G/\ll N\rr^G$ is hyperbolic relative to
$\Hl\cup\{E(g)/N\}$.
\item[(b)]
The natural homomorphism $G\to G/\ll N\rr$ is injective on $X$ (in
particular, on $\mathcal H$).
\end{enumerate}

Since $E(g)$ is virtually cyclic, there exists $n\in\NN$ such that the
subgroup $N=\langle g^n\rangle $ is normal in $E(g)$ and avoids
$\mathcal F$. By (a), the group $Q=G/\ll g^n\rr$ is hyperbolic relative to
$\Hl\cup\{E(g)/N\}$. Note that $E(g)/N$ is finite and, therefore, can be
excluded from the collection of peripheral subgroups by \cite[Theorem
2.40]{Osi06a}. Thus, $Q$ is hyperbolic relative to $\Hl$. Obviously, the
image of $g$ has finite order in $Q$.
\end{proof}

\begin{cor}\label{Cor:T}
Suppose that a countable group $G$ is hyperbolic relative to a collection
of periodic subgroups $\Hl$. Then there exists a periodic quotient group
$T$ of $G$ such that the restriction of the natural homomorphism $G\to T$
to $\bigcup_{\lambda\in \Lambda} H_\lambda$ is injective.
\end{cor}

\begin{proof}
We fix an arbitrary enumeration $G=\{g_0=1, g_1, g_2, \ldots\}$ of elements
of $G$. By induction, we will construct a sequence of groups and
epimorphisms
$$
G\stackrel{\gamma_0}\longrightarrow
G_0\stackrel{\gamma_1}\longrightarrow G_1
\stackrel{\gamma_2}\longrightarrow\ldots
$$
satisfying the following
conditions
for all $i\in \NN \cup\{ 0\}$.

\begin{enumerate}
\item[(a)] The composition
$\varkappa_i=\gamma_i\circ \cdots \circ \gamma_{0}$ is injective on the set
$\bigcup_{\lambda\in \Lambda} H_\lambda$. In what follows,
we identify each $H_\lambda$ with its image in $G_i$.
\item[(b)] $G_i$ is hyperbolic relative to $\Hl$.
\item[(c)] The element $\varkappa_i(g_i)$ has finite order in $G_i$.
\end{enumerate}

Let $G_0=G$ and let $\gamma_0$ be the identity map. Clearly,
(a), (b), and (c)
are
satisfied for $i=0$.  Further, suppose that we have already constructed
groups $G_0, \ldots, G_i$ and the corresponding epimorphisms satisfying
the inductive assumption. At step $i+1$, the required group $G_{i+1}$ and
homomorphism $\gamma_{i+1}\colon G_i\to G_{i+1}$ exist by Lemma
\ref{Lem:Q} applied to the element $\varkappa_i(g_{i+1})\in G_i$. Note
that the assumption that all groups $H_\lambda$ are periodic is used here to ensure that
if the image of $g_{i+1}$ is conjugate to an element of some $H_\lambda$,
then it has finite order.

Let
$$
T=G/\bigcup_{i\in \NN} \ker \varkappa _i.
$$
By (a), the natural homomorphism $G\to T$ is injective
on $\bigcup_{\lambda\in \Lambda} H_\lambda$, while (c) ensures that $T$
is periodic.
\end{proof}

Let $G$ be a group splitting as $G=A\ast_C B$, where $C$ is a finite group.
It is well known that $G$ is hyperbolic relative to $\{ A, B\}$.
This result is commonly attributed to Bowditch since the action of $G$ on the associated Bass-Serre
tree satisfies his definition of relative hyperbolicity as stated in
\cite[Definition~3.4]{Hru}.
(The original version of the definition given in \cite{Bow}
assumed finite generation of the peripheral
subgroups, which is unnecessary).
Note that all known definitions of relative hyperbolicity are equivalent
for countable groups \cite{Hru}. Therefore,
every countable group $G=A\ast_C B$
is hyperbolic relative to $\{ A, B\}$ in the sense of the definition given
in \cite{Osi06a},
which we use in this paper. In fact, countability is irrelevant here;
indeed, relative hyperbolicity of $G$ can be derived directly from
the definition given
in \cite{Osi06a}.
For the reader familiar with the terminology introduced therein,
we provide a brief argument.

\begin{lem}
The free product $A*_CB$ of any groups $A$ and $B$ with finite
amalgamated subgroup $C$ is hyperbolic relative to the pair
of subgroups $\{A,B\}$.
\end{lem}

\begin{proof}
Let $c\mapsto\~c$ and $c\mapsto\^c$ denote
embeddings of $C$ into $A$ and $B$, respectively.
Consider the natural homomorphism
$\e\colon A\ast B\to A*_CB$. Clearly,
$\ker \e$ coincides with the normal closure
of the finite set $\{(\^c)^{-1}\~c\;|\;c\in C\}$ in ${A\ast B}$.
Let $w=d_1d_2\dots d_n\in \ker\e$ be a nonidentity element,
where every $d_i$ belongs to $A\sqcup B$
and $n$ is as small as possible for the given $w$
(i.e., $d_i\ne 1$ and
$d_i$ lying in $A$ alternate with $d_i$ lying in~$B$).
By the normal form
theorem for amalgamated free products
(see, e.g., \cite{LS80}), we have $d_i=\~c$
or $d_i=\^c$ for some~$i$ and some $c\in C$. In the former case,
we obtain  the equality
$$
w=
\underbrace{
d_1\dots d_{i-2}
\overbrace{(d_{i-1}\^cd_{i+1})}^{d_i'}
d_{i+2}\dots d_n
}_{w'}
\cdot\bigl((\^c)^{-1}\~c\bigr)^{d_{i+1}\dots d_n}
$$
in $A*B$, where the length
of the word $w'$ in the alphabet $A\sqcup B$ is less than the length of $w$
(which is $n$),
and similarly in the latter case.
The obvious induction shows that
$w$ can be represented as a product of at most $n$
conjugates of elements from the set $\{(\^c)^{-1}\~c\;|\;c\in C\}$. This
means that the relative Dehn function of $G$ with respect to $\{ A, B\}$
is linear,
and the lemma follows.
\end{proof}

\begin{proof}[Proof of Theorem \ref{Thm:AmEmb}]
By the previous lemma, $A\ast_C B$ is hyperbolic relative
to $\{ A, B\}$. By Corollary~\ref{Cor:T}, there exists a torsion quotient
group $Q$ of $A\ast_CB$ such that the natural homomorphism $A\ast_CB\to Q$
is injective on $A\cup B$.
\end{proof}

\section{Applications}\label{Sec:App}

In this section, we discuss applications of the results obtained above to
some particular classes of groups. We will write
\begin{tikzcd}
\K \arrow[r, "e" description, hook] & \L
\end{tikzcd}
to mean that a
subclass $\K$ is
equationally
separable in a class of groups $\L$. If, in addition, the
separating system can be made univariate (respectively, can be chosen to
consist of a single equation) for every $G\in \K$, we add $u$
(respectively, $1$) to the arrow
label. These additional labels can be added individually or simultaneously;
for example,
the label $ue1$ means that the separating system can be chosen to consist
of a single univariate equation for every $G\in \K$.
Finally, the addition of
$\not\hspace{-.5mm} 1$ to the arrow label $e$ or $ue$ between $\K$ and
$\L$ means that
there exists a nontrivial group $G\in \K$ such that the corresponding
separating system cannot consist of a single equation.

\begin{thm}\label{Thm:main}
The following relations between classes of groups hold.
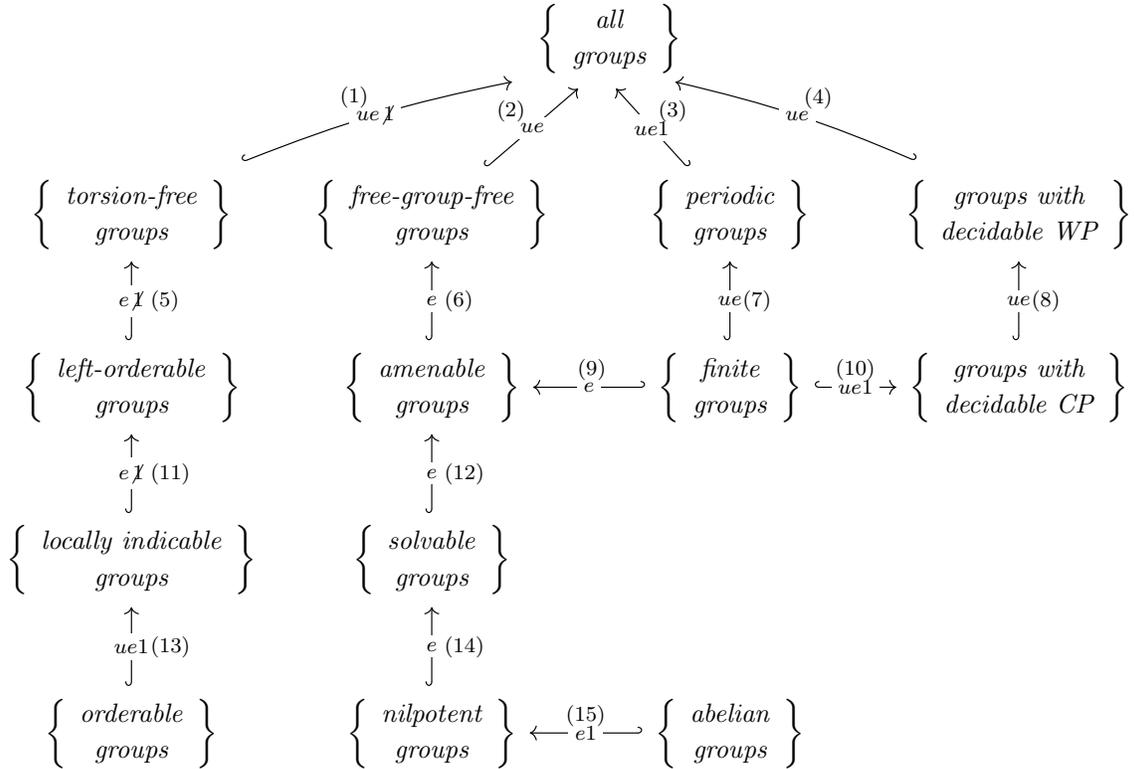
\begin{center}
\vbox{
\hspace*{12mm}\begin{tikzcd}
\left\{\begin{array}{c} \text{\small all}\\
\text{\small groups}\end{array}\right\}
\end{tikzcd}
\vspace{-4mm}

\begin{tikzcd}[row sep=large]
&[-5mm] {}&[-21mm]{\hspace{30mm}}&[-21mm]{}&\\
\left\{\begin{array}{c} \text{\small torsion-free}\\
\text{\small groups}\end{array}\right\}
\arrow[rru, "ue\not{\, 1}" description, "(1)", hook, bend left=5,
shift left=1] &[-5mm]
\left\{\begin{array}{c}  \text{\small free-group-free} \\
\text{\small groups} \end{array}\right\}
\arrow[ru, "ue" description, "(2)",hook] &[-21mm]&[-21mm]
\left\{\begin{array}{c}  \text{\small periodic}\\
\text{\small groups}\end{array}\right\}
\arrow[lu, "e1" description, "(3)"', hook'] &
\left\{\begin{array}{c}  \text{\small groups with}\\
\text{\small decidable WP}\end{array}\right\}
\arrow[llu, "ue" description, "(4)"', hook', bend right=5, shift right=1]
\\
\left\{\begin{array}{c}  \text{\small left-orderable}\\
\text{\small groups}\end{array}\right\}
\arrow[u, "e\not{\, 1}" description, "\;\;(5)"', hook] &[-5mm]
\left\{\begin{array}{c}  \text{\small amenable}\\
\text{\small groups}\end{array}\right\}
\arrow[u, "e" description, "\; (6)"',  hook]  &[-21mm]&[-21mm]
\left\{\begin{array}{c}  \text{\small finite}\\
\text{\small groups}\end{array}\right\}
\arrow[u, "ue"
description, "\; (7)"', hook] \arrow[r, "e1" description, "(10)", hook]
\arrow[ll, "e" description, "\; (9)"', hook'] &
\left\{\begin{array}{c}  \text{\small groups with} \\
\text{\small decidable CP} \end{array}\right\} \arrow[u, "ue"
description, "\; (8)"', hook]
\\
\left\{\begin{array}{c}  \text{\small locally indicable}\\
\text{\small groups}\end{array}\right\}
\arrow[u, "e\not{\, 1}" description, "\;\; (11)"', hook]    &[-5mm]
\left\{\begin{array}{c}  \text{\small solvable}\\
\text{\small groups}\end{array}\right\}
\arrow[u, "e" description, "\; (12)"', hook]  &[-21mm]&[-21mm] &
\\
\left\{\begin{array}{c}  \text{\small orderable}\\
\text{\small groups}\end{array}\right\}
\arrow[u, "ue1" description,"\;\; (13)"', hook] &[-5mm]
\left\{\begin{array}{c}  \text{\small nilpotent}\\
\text{\small groups}\end{array}\right\}
\arrow[u, "e" description, "\; (14)"', hook]
&[-21mm]&[-21mm] \left\{\begin{array}{c}
\text{\small abelian}\\ \text{\small groups}\end{array}\right\}
\arrow[ll
, "e1" description, "(15)"',hook'] &
\end{tikzcd}
}
\end{center}
Furthermore,
for any set of primes $\rho$ and any proper subset $\pi\subset \rho$,
the following refinement of (7) holds:
\begin{center}
\begin{tikzcd}
\left\{ \begin{array}{c}
\text{\small $\pi $-groups}\end{array}\right\}
\arrow[r, "e1" description, "(16)", hook]&
\left\{ \begin{array}{c} \text{\small $\rho$-groups}\end{array}\right\}  &
\left\{ \begin{array}{c} \text{\small finite}\\
\text{\small $\rho$-groups}\end{array}\right\}
\arrow[l, "e" description, "(17)"', hook'] &
\left\{ \begin{array}{c} \text{\small finite}\\
\text{\small $\pi$-groups}\end{array}\right\}
\arrow[l, "e1" description, "(18)"', hook']
\end{tikzcd}
\end{center}
\end{thm}

The numbers in parentheses next to arrows in Theorem \ref{Thm:main} have
no mathematical meaning; they are added to designate separate claims and
serve for reference purposes below. Abbreviations CP and WP denote the
conjugacy and word problems, respectively.

To prove Theorem \ref{Thm:main}, we will need a couple of known facts.

\begin{lem}[Baumslag, \cite{Ba69}]
The element $a$ of the group $\left\la a,b \mid a^{a^b}=a^2\right\ra$
is nontrivial and contained in the kernel of every group homomorphism
sending $a$ to a finite-order element.
\end{lem}

An equation $w(x,y,\dots)=1$ over a group $G$
is called \emph{non-exotic} if $w(x,y,\dots)$ is not conjugate to an
element of $G$ in $G*F(x,y,\dots)$.

\begin{thm}[Brodskii--Howie--Short \cite{B80,B84,How81,Sh81}]
Any non-exotic equation over a locally indicable group
is solvable in a
locally indicable overgroup.
\end{thm}

We now proceed with the proof of Theorem \ref{Thm:main}, which will be
subdivided into four separate themes.

\bigskip

\noindent{\bf Separability by a single equation.}

\emph{Claim (3). }  We first note that, for any
group $G$ with torsion, there exists a non-exotic equation with
coefficients in $G$ unsolvable over this group. Indeed, let
$a\in G\setminus \{ 1\}$ be a finite-order element.
If $a^2\ne 1$, we can take
Baumslag's equation $a^{a^x}=a^2$. If $a^2=1$, we can take the equation
$a^{x^2}=[a,a^x]$, whose unsolvability over the group
$\la a\ra_2\cong \ZZ_2$ is proved by Lyndon in \cite{Ly80}.

Let $G$ be a periodic group. We consider the equation
\begin{equation}\label{*}
\Big(w(x)\Big)^{\left(w(x)\right)^y}=(w(x))^2,
\end{equation}
where $w(x)=1$ is a non-exotic equation with coefficients in $G$
unsolvable over $G$. If (\ref{*}) has a solution $(\tilde x,\tilde y)$ in
a periodic group containing $G$, then, by virtue of Baumslag's lemma, we
have $w(\~x)=1$. This contradicts unsolvability of the equation $w(x)=1$
over $G$. Thus, (\ref{*}) is unsolvable in the class of periodic groups.

We next show that (\ref{*}) is solvable over $G$. Since the
equation $w(x)=1$ is non-exotic, the element $w(c)$ of the group
$G*\la c\ra_\infty$ is not conjugate to an element of $G$ and, therefore,
has infinite order. The element $a\in\left\la a,b \mid
a^{a^b}=a^2\right\ra$ also has infinite order by
Baumslag's
lemma.
Hence, we can form the amalgamated free product \begin{equation}\label{FP}
\widetilde G=\big(G*\la c\ra _\infty\big)\ast_{w(c)=a}\left\la a,b \mid
a^{a^b}=a^2\right\ra. \end{equation} Clearly, $x=c$ and $y=b$ give
a solution
to (\ref{*}) in $\widetilde G$.

\medskip

\emph{Claim (10)}.
Let $G$ be a finite group. We will show that the same equation (\ref{*})
works in this case. It suffices to prove that the group $\widetilde G$
defined by (\ref{FP}) has decidable conjugacy problem.

As in the proof of Claim (3),
choose $g\in G\setminus\{1\}$ and a non-exotic equation
$w(x)=1$ with coefficients in
$\gp g$ that is unsolvable over $\gp g$.
Since the equation is non-exotic, $w(c)\in \gp g*\gp c_\infty$
is not conjugate to any element of $\gp g$ and, hence, is of infinite
order.
In addition, we may assume that $w(c)$ is not a proper power
(because, if $v(x)^k=1$ is a non-exotic unsolvable equation,
then so is $v(x)=1$).

Moreover, the maximal virtually cyclic
subgroup~$E\bigl(w(c)\bigr)\subset\gp g*\gp c_\infty$
containing $w(c)$ coincides with $\gp{w(c)}$.
Indeed,
$E\bigl(w(c)\bigr)\subset\gp g*\gp c_\infty$
is either cyclic or infinite dihedral by the Kurosh subgroup theorem.
Suppose it is dihedral. Then $w(c)$ is a product of two involutions:
$u(c)$ and $v(c)$;
hence, $w(1)\in\gp g\setminus\{1\}$
(it is nontrivial since the equation is unsolvable)
is a product of two elements of order at most two:
$u(1)\in\gp g$ and $v(1)\in\gp g$;
therefore, $u(1)=1$ or $v(1)=1$ because $\gp g$ is cyclic.
Say $u(c)\in\nc c$; but $\nc c$ is a free group
(by the Kurosh subgroup theorem) and contains no involutions.
This is a contradiction. Thus, $E\bigl(w(c)\bigr)$ is cyclic
and, hence, coincides with $\gp{w(c)}$ (because $w(c)$ is not a proper
power).

We can now apply \cite[Corollary 1.7]{Osi06b} to conclude that
$\gp g*\gp c_\infty$
is hyperbolic relative to $E\bigl(w(c)\bigr)=\gp{w(c)}$.
Applying Lemma \ref{Lem:CT}, we obtain
that
$$
\^G=(\gp g*\gp c_\infty)*_{w(c)=a}\pres<a,b|a^{a^b}=a^2>
$$
is hyperbolic relative
to Baumslag's group
$\pres<a,b|a^{a^b}=a^2>$.
The conjugacy problem in Baumslag's
group
is decidable \cite{Be12} (see also \cite{DMW16}). By
the main result of \cite{Bu04}, \emph{any finitely generated group
hyperbolic relative to subgroups with decidable conjugacy problem has
decidable conjugacy problem itself}. Therefore, $\^G$ has
decidable conjugacy problem.

The conjugacy problem in $\~G=\^G*_{\gp g}G$
is decidable too because of
the following well-known fact \cite[Chapter 4, Exercise 40]{MKS66}:
\disp{\sl
If groups $A$ and $B$ have decidable conjugacy problem,
then so does $A*_CB$ for finite $C$.
}
Thus, the equation $w(x)^{w(x)^y}=w(x)^2$
has no solutions in any finite overgroup of $G$,
but has a solution ($x=c$ and $y=b$) in the group $\~G$,
which has decidable conjugacy problem;
and the claim follows.

\medskip

\emph{Claim (13).}
Orderable groups are locally indicable \cite{RR02}. The equation
$g^x=g^{-1}$, where $g\ne 1$ is an element of an orderable group $G$,
cannot have solutions in an orderable group (if $g>1$ in an orderable
group, then $g^x>1$ and $g^{-1}<1$), but has a solution in a locally
indicable group by the Brodskii--Howie--Short theorem.

\medskip

\emph{Claim (15)}.
Let $g$ be a nontrivial element of an abelian group $G$.
Then the equation $[x,y]=g$ cannot have solutions in abelian groups
containing $G$, but has an obvious solution
($x=a$, $y=b$)
in the central product~$(H\times G)/\la g^{-1}c\ra$,
where
$$
H=\la a,b,c|[a,b]=c,\;[a,c]=[b,c]=c^n=1\ra
$$
is the quotient of the Heisenberg
group $UT_3(\ZZ)$ by the subgroup $\la c^n\ra$ and~$n$ is the order of $g$
(or zero, if
$|\gp g|=\infty$).

\medskip

\emph{Claim (16).}
If $\pi=\emptyset$, then we have nothing to prove. Therefore, we can
assume that $\pi$ is non-empty.  In a nontrivial $\pi$-group $G$, we pick
an element $a$ of prime order $p\in\pi$ and choose a prime $q\in\rho
\setminus\pi$ and $k\in\NN$ in such a way that $q^k>p$. The equation
$
(ax)^{q^k}=x^{q^k}
$
cannot have a solution in a $\pi$-group since
the map $g\mapsto g^{q^k}$ is bijective in every $\pi$-group.

It remains to show that the equation $(ax)^{q^k}=x^{q^k}$
is solvable in some $\rho$-group $\widetilde G\supset G$.
To this end, let $\widetilde G=G\wr\la \~x\ra_{q^k}$.
The group $G$ embeds in the base of this wreath product
by the map
$$
g\mapsto (\underbrace{g,g,\ldots,g}_{p\ entries},
\underbrace{1,1,\ldots,1}_{q^k-p\ entries}).
$$
The element $\~x\in\widetilde G$ is a solution to the equation
$(ax)^{q^k}=x^{q^k}$ as both sides evaluate to $1$ at this element. This
completes the proof.

\medskip

\emph{Claim (18).}
It suffices to note that the group $\widetilde G$ from the proof of Claim
(16)
is finite whenever $G$ is finite.

\paragraph{\bf Inseparability by a single equation}
in (1), (5), and (11) follows immediately from the Brodskii--Howie--Short
theorem (see above).

\paragraph{Separability by finite systems of equations}
of the subclass $\K$ in the class $\L$ in all the remaining cases is a
consequence of Theorem \ref{Thm:Simple}, the simplicity of the
corresponding class $\L$ (see Proposition \ref{Prop:Simple}), and the
following quasi-identities satisfied in $\K$ but not in $\L$.

\emph{Claim (1).}
The quasi-identity $x^2=1  \,\Rightarrow \,  x=1$ holds in all
torsion-free groups, but not in all groups.

\emph{Claim (2).}
By the Boone--Higman theorem, the free group of rank $2$ embeds into a
simple subgroup~$S$ of a finitely presented group $H$. Applying Lemma
\ref{Lem:qi} to $Q=H$ and any nontrivial $q\in S$ yields the required
quasi-identity.

\emph{Claim (4).} As shown in \cite{Mil81},
there exists a finitely presented group $M$ having no nontrivial
homomorphisms to a group with decidable word problem. Applying Lemma
\ref{Lem:qi} to $Q=M$ and any nontrivial element $q\in M$, we obtain the
required quasi-identity.

\emph{Claim (5).}
The group $\la x,y\mid y^{2x}=y^{-2},\;x^{2y}=x^{-2}\ra $ is torsion-free,
but is not left-orderable (it is not even a unique-product group), and its
quotient group by any nontrivial proper normal subgroup has torsion
\cite{Pr88}. Thus, the quasi-identity $(y^{2x}=y^{-2}\;\&\;x^{2y}=x^{-2})
\,\Rightarrow \, x=1$ is satisfied in all left-orderable groups, but not
in all torsion-free groups.

\emph{Claim (6).}
By \cite{LM}, there exists a finitely presented, non-amenable,
free-group-free group $Q$ (denoted by $G_0$ in \cite{LM});
furthermore, the derived subgroup $Q^{\prime}$ is known to be
simple \cite{BLR}.
Note that $Q^{\prime}$ must be non-amenable as otherwise $Q$ would be
amenable.
Therefore, $Q^{\prime}$ is contained in the kernel of every homomorphism
from $Q$
to an amenable group. Thus, Lemma \ref{Lem:qi} can be applied to $Q$
and any nontrivial element $q\in Q^{\prime}$.

\emph{Claim (7).}
This is a special case of Claim (17), which is proved below.

\emph{Claim (8).}
Let $D$ be a finitely presented group with decidable word problem that
cannot be embedded into a group with decidable conjugacy problem; such a
group exists by \cite{Dar21}. By \cite[Theorem 4.11]{Th80}, $D$ embeds
into a simple subgroup of a finitely presented group
$T=\la x_1,\dots,x_n\mid w_1=1,\dots,\;w_m=1\ra $ with decidable word
problem.  Applying Lemma \ref{Lem:qi} to $Q=T$ and any nontrivial element
$q\in D$ yields the required quasi-identity.

\emph{Claim (9).}
There exists a non-residually finite, finitely presented solvable (hence
amenable) group $B$ \cite{Ba73}. Applying Lemma \ref{Lem:qi} to $Q=B$ and
any nontrivial element in the intersection of kernels of all
homomorphisms from $B$ to finite groups, we obtain the required
quasi-identity.

\emph{Claim (11).}
The group $\la x,y,z\mid x^2=y^3=z^7=xyz\ra $ coincides with
its commutator subgroup, is left-orderable, and is nontrivial ($x\ne1$)
\cite{Ber91}.
Thus, the quasi-identity $x^2=y^3=z^7=xyz\,\Rightarrow\, x=1$ holds in all
locally indicable groups but not in all left-orderable groups.

\emph{Claim (12).}
It suffices to apply Lemma \ref{Lem:qi} to the (amenable) alternating
group $A_5$.

\emph{Claim (14).}
The quasi-identity $x=[x,y]\,\Rightarrow\, x=1$ obviously holds in all
nilpotent groups, but does not hold in the solvable (even metabelian)
Baumslag--Solitar group
$$
BS(1,2)= \la x,y\mid x^y=x^2\ra = \la x,y\mid
x=[x,y]\ra.
$$

\emph{Claim (17).}
For $\rho \ne\{2\}$ the required quasi-identity exists by the following.

\begin{thm}[Ivanov \cite{I05}]
For any odd $n>2^{16}$ and any class $\K$ such that all finitely generated
subgroups of period $n$ of groups from $\K$ are finite, there exists an
infinite finitely generated group of period $n^3$ that does not belong to
the quasivariety generated by $\K$.
\end{thm}

Indeed, it suffices to take $\K=\{\text{finite $\rho$-groups}\}$ and
$n=p^{16}$, where $p\ne 2$ is any prime in $\rho$. Note that
$\rho$ is nonempty since it contains the proper subset $\pi$.

If $\rho=\{2\}$, we take a minimal finite-index subgroup $G$ in the free
Burnside $2$-generated group of period $2^k$, where $k\gg 1$ (such a
minimal subgroup exists by \cite{Z91}). The group $G$ is finitely
generated (because it is a finite-index subgroup of a finitely generated
group), nontrivial (because the free Burnside group is infinite for
sufficiently large $k$ \cite{I94, Ly96}), and coincides with its
commutator subgroup (by minimality). Let $G=\la x_1,\dots,x_n\ra$, where
$x_1\ne 1$, and let $x_i=w_i(x_1,\dots,x_n)$, where the words $w_i$ lie in
the commutator subgroup of the free group with basis
$\{ x_1, \ldots, x_n\}$. The quasi-identity
$$
\big(x_1=w_1(x_1,\dots,x_n)\;\&\dots\;\&\;x_n=w_n(x_1,\dots,x_n) \big) \,
\Rightarrow\, x_1=1
$$
does not hold in $G$, but holds in all finite $2$-groups because a finite
$p$-group coinciding with its commutator subgroup is trivial.

\paragraph{Separability by univariate systems.}
In \emph{Claim (13)}, the separating
equation provided above is
already
univariate. To deal with other cases, we introduce the following.

\begin{defn}
A class $\L$ of groups is called \emph{RH-closed} if every group
hyperbolic relative to a group from $\L$ belongs to $\L$ itself.
\end{defn}

\begin{lem}\label{Lem:red}
Suppose that $\K$ is an equationally separable subclass of an RH-closed
class of groups $\L$. Then
\begin{tikzcd}
\K \arrow[r, "ue" description, hook] & \L.
\end{tikzcd}
\end{lem}
\begin{proof}
Let $G\in \K$ and let $\mathcal A$ be a separating system that has a
solution in some group $\widetilde G\in \L$ containing $G$ but is
unsolvable in $\K$. Let $\mathcal B$ be the system provided by Proposition
\ref{Prop:reduction}. Since $\mathcal B$ is obtained from $\mathcal A$ by
a substitution, every solution to $\mathcal B$ in a group yields a
solution to $\mathcal A$ in the same group. This implies that $\mathcal B$
is unsolvable in $\K$. On the other hand, Proposition~\ref{Prop:reduction}
guarantees that $\mathcal B$ is solvable in a group hyperbolic relative to
$\widetilde G$, which belongs to $\L$ by our assumption.
\end{proof}

We now return to the proof of separability by univariate systems in the
remaining claims of Theorem \ref{Thm:main}.

\emph{Claims (1)--(4)}.
The class $\L$ of all groups is obviously RH-closed and Lemma
\ref{Lem:red} applies.

\emph{Claims (8) and (10).}
The classes of groups with decidable WP and CP are RH-closed by \cite{F}
and \cite{Bu04} respectively, so we can apply Lemma \ref{Lem:red} again.

\emph{Claim (7).}
Let $G$ be a finite group and let $\mathcal A$ be a separating system that
has a solution in some periodic group $\widetilde G$ containing $G$ but is
unsolvable in the class of finite groups. Let $\mathcal B$ be the system
provided by Proposition \ref{Prop:reduction}. Arguing as in Lemma
\ref{Lem:red}, we obtain that $\mathcal B$ is unsolvable in the class of
finite groups. On the other hand, by Proposition \ref{Prop:reduction},
$\mathcal B$ is solvable in a group $R$ hyperbolic relative to
$\widetilde G$. By Corollary \ref{Cor:T}, there is a homomorphism
$\gamma\colon R\to T$, where $T$ is a periodic group, such that the
restriction of $\gamma $ to $\widetilde G$ is injective. Identifying
$\widetilde G$ with its image in $T$,  and mapping the solution to the
system $\mathcal B$ in $R$ to the group $T$ via $\gamma$, we obtain that
$\mathcal B$ has a solution in a periodic group containing $\widetilde G$.

\vspace{5mm}

\noindent \textbf{Alexander A. Buturlakin: }
Sobolev Institute of Mathematics,
Novosibirsk 630090, prospekt Akademika Koptyuga, 4.

\noindent E-mail: \emph{buturlakin@gmail.com}

\bigskip

\noindent \textbf{Anton A. Klyachko: }
Faculty of Mechanics and Mathematics of Moscow State University,
Moscow 119991, Leninskie gory, MSU.

\noindent
Moscow Center for Fundamental and Applied Mathematics.

\noindent E-mail: \emph{klyachko@mech.math.msu.su}

\bigskip

\noindent \textbf{Denis V. Osin: }
Department of Mathematics, Vanderbilt University, Nashville 37240, U.S.A.

\noindent E-mail: \emph{denis.v.osin@vanderbilt.edu}

\end{document}